\documentclass{amsart}

\usepackage{mathtools}
\usepackage{xypic}
\usepackage{amsthm}
\usepackage{amsmath}
\usepackage{amssymb}
\usepackage{mathrsfs}
\usepackage{cite}
\usepackage{graphicx}

\newtheorem{thm}{Theorem}
\newtheorem{question}[thm]{Question}

\newtheorem{prop}[thm]{Proposition}

\newtheorem{remark}[thm]{Remark}

\newcommand{\nc}{\newcommand}
\nc{\rnc}{\renewcommand}
\rnc{\P}{\mathbb P}
\nc{\R}{\mathbb R}
\nc{\C}{\mathbb C}
\nc{\A}{\mathbb A}
\nc{\Q}{\mathbb Q}
\nc{\Z}{\mathbb Z}
\rnc{\O}{\mathcal O}
\nc{\LL}{\mathbb L}
\nc{\Hom}{\text{Hom}}
\nc{\codim}{\text{codim}}
\nc{\Sym}{\text{Sym}}
\nc{\Spec}{\text{Spec}\,}
\nc{\End}{\text{End}}
\nc{\eps}{\epsilon}
\nc{\Pic}{\text{Pic}}
\nc{\ov}{\overline}
\nc{\F}{\mathfrak F}
\nc{\G}{\mathbb G}
\nc{\E}{\mathcal E}
\rnc{\S}{\mathcal S}

\rnc{\L}{\mathcal L}
\rnc{\H}{\mathbb H}
\nc{\Bl}{\text{Bl}}
\nc{\im}{\text{im}}
\nc{\gr}{\text{gr}}

\nc{\m}{\mathfrak m}

\nc{\us}{\underset}
\nc{\ul}{\underline}
\nc{\bs}{\backslash}
\nc{\os}{\overset}

\nc{\Tot}{\text{Tot}}
\nc{\Mod}{\text{Mod}}
\nc{\QMod}{\text{QMod}}
\nc{\MW}{\text{MW}}
\nc{\NS}{\text{NS}}
\nc{\Res}{\text{Res}}
\nc{\Aut}{\text{Aut}}
\nc{\W}{\mathcal W}
\nc{\NL}{\text{NL}}
\nc{\mult}{\text{mult}}

\nc{\U}{\mathscr U}
\nc{\D}{\mathscr D}
\nc{\M}{\mathcal M}

\rnc{\v}{{\langle v \rangle}}
\nc{\Jac}{\text{Jac}}
\nc{\J}{\ov{\mathcal J}}

\title{A Lagrangian sphere which is not a vanishing cycle}

\author{Fran\c{c}ois Greer}

\begin{document}
\maketitle
\begin{abstract}
We give examples of Calabi-Yau threefolds containing Lagrangian spheres which are not vanishing cycles of nodal degenerations, answering a question of Donaldson in the negative.
\end{abstract}

\section{Introduction}
\noindent
The $n$-dimensional nodal singularity has a 1-parameter versal deformation
$$\left(\sum_{i=0}^n z_i^2 = t\right) \subset \C^{n+1}\times \C_t.$$
The nearby fiber over $t=\eps>0$ retracts onto an $n$-sphere, the vanishing cycle:
$$S^n\simeq \left(\sum_{i=0}^n x_i^2 = \eps\right),$$
which is Lagrangian with respect to the standard symplectic form $\omega = \sum dx_i\wedge dy_i$ ($z_i = x_i+y_i\sqrt{-1}$).  A natural question, first raised by Donaldson \cite{donaldson}, is whether all Lagrangian spheres arise in this way.
\begin{question}
Let $Z$ be a complex projective manifold, and $L\subset Z$ embedded sphere, Lagrangian with respect to a K\"{a}hler form on $Z$.  Is $L$ always the vanishing cycle of a nodal degeneration of $Z$?
\end{question}\label{donaldson}
\noindent
The answer is positive for curves, and unknown for surfaces.  For certain K3 surfaces, the answer is positive \cite{sheridan}, modulo the issue of whether Fukaya isomorphism implies Hamiltonian isotopy.  For Horikawa surfaces, a positive answer would distinguish two particular deformation types as smooth manifolds \cite{auroux}.  We show that the answer to Question \ref{donaldson} is negative in general:
\begin{thm}
There exists a rigid projective Calabi-Yau threefold $\hat{X}$ with a Lagrangian sphere $L\subset \hat{X}$ which is homologically non-trivial (essential).
\end{thm}
\noindent
Rigidity implies that any degeneration of $\hat{X}$ is isotrivial.  We prove further that such a degeneration must have monodromy of order $\leq 6$ on $H_3(\hat{X})$.  In particular, this rules out nodal degenerations with vanishing cycle $L$; their monodromy would be a Dehn twist by $[L]\in H_3(\hat{X})$, which has infinite order.  It was known \cite{st} that if an essential Lagrangian sphere existed on a rigid CY3, then it could not be the vanishing cycle of a nodal degeneration.  Our construction produces four non-isomorphic rigid Calabi-Yau threefolds with Euler characteristics $66$, $72$, $80$, and $96$, each containing essential Lagrangian spheres.\\\\
{\bf Acknowledgments.}  The author is grateful to Mark McLean for suggesting the problem, and to the anonymous referees for valuable comments.  He has also benefitted from communications with Denis Auroux, Jonathan Evans, Ivan Smith, Richard Thomas, and Abigail Ward.
\section{The Construction}
\noindent
The counterexamples are among the Calabi-Yau threefolds considered by Schoen \cite{schoen}.  Consider the following pencil of cubics in $\P^2$:
$$(x+y)(y+z)(z+x) + t xyz =0.$$
Viewed as a family of curves over $\P^1$, the relatively minimal smooth model 
$$\nu:S\to \P^1$$
has four singular Kodaira fibers of types $I_6$, $I_3$, $I_2$, and $I_1$ over $t=\infty, 0, 1$, and $-8$, respectively.  This is one of six semistable elliptic families over $\P^1$ (all extremal) with the minimum number of singular fibers, as constructed by Beauville \cite{beauville}, and it is isomorphic to the universal family over the compactified modular curve $X_1(6)$.  See \cite{herfurtner} for a comprehensive reference on elliptic surfaces with four singular fibers.\\\\
Let $\phi$ be a non-trivial automorphism of $\P^1$ which permutes $\{\infty,0,1\}$, and note that $\phi(-8)\neq -8$.  We form the Cartesian product
$$
\xymatrix{
X \ar[r]\ar[rd]^\pi \ar[d] & S \ar[d]^{\phi\circ\nu} \\
S \ar[r]^{\nu}  & \P^1.
}
$$
The result is a singular projective threefold with $K_X=0$, fibered over $\P^1$ by abelian surfaces which are products of non-isogenous elliptic curves.  There are five critical values:  $\infty, 0, 1, -8,$ and $\phi(-8)$.  The total space $X$ has $n$ conifold singularities in the fibers over $\{\infty, 0,1\}$, located at the product of two nodes in the elliptic fibers.  For the different choices of $\phi$, we get $n=33,36,40,48$.  The value $n=36$ arises when $\phi$ acts as a 3-cycle; using either 3-cycle produces the same singular threefold.  The remaining values of $n$ arise from transpositions.\\\\
There exists a projective small resolution $\eps:\hat{X}\to X$ obtained by successive blow ups of the $n$ Weil divisors which are irreducible components of the singular fibers.  The resolution is crepant, so $\hat{X}$ is a smooth Calabi-Yau threefold.  We will use $\pi$ interchangeably with $\pi\circ \eps$ when no confusion can arise.
\begin{prop}\label{pic}
The Picard group of $\hat{X}$ has rank $n$.
\end{prop}
\begin{proof}
The specialization of $\pi:\hat{X}\to \P^1$ to the generic point $\eta\in \P^1$ gives a split short exact sequence
$$0 \to A \to \Pic(\hat{X}) \to \Pic(X_\eta) \to 0,$$
where $A$ is the span of the $n$ divisor classes supported over $\{\infty,0,1\}$.  They satisfy 2 relations using the rational equivalence over $\P^1$.  The generic fibers of $\nu$ and $\phi\circ\nu$ are non-isogenous elliptic curves, so 
$$\Pic(X_\eta) \simeq \Pic(S_\eta)\oplus \Pic(S_\eta).$$
The specialization of $\nu:S\to \P^1$ to $\eta$ gives a split short exact sequence
$$0 \to B \to \Pic(S) \to \Pic(S_\eta)\to 0,$$
where $B$ is the span of the $12$ curves classes supported over $\{\infty,0,1,-8\}$.  They satisfy 3 relations using the rational equivalence over $\P^1$.  Since $\rho(S)=10$, we find that $\Pic(S_\eta)$ has rank 1.  In fact, the torsion Mordell-Weil group of $S_\eta$ is known \cite{mirandabasic}:
$$\Pic(S_\eta) \simeq \Z\oplus \Z/6\Z.$$
\end{proof}
\begin{prop}\label{rigid}
The threefold $\hat{X}$ is rigid in the sense that $H^1(T_{\hat{X}})=0$.
\end{prop}
\begin{proof}
The cup product gives an isomorphism $H^1(T_{\hat{X}}) \simeq H^{1,2}(\hat{X})^\vee$ since $\hat{X}$ is Calabi-Yau.  As $X$ is fibered by tori, its topological Euler characteristic is determined by the singular fibers $I_b\times I_{b'}$.  This gives
\begin{align*}
\chi_{top}(X) &= n\\
\chi_{top}(\hat{X}) &= 2n,
\end{align*}
since the small resolution replaces each conifold point with a $\P^1$.  On the other hand, the Calabi-Yau property gives
$$\chi_{top}(\hat{X}) = 2\left(h^{1,1}(\hat{X}) - h^{1,2}(\hat{X})\right),$$
and $H^{1,1}(\hat{X})\simeq \Pic(\hat{X})$, so we are done by Prop. \ref{pic}.
\end{proof}
\begin{prop} \label{drag} $\hat{X}$ contains an embedded 3-sphere $L$ that is Lagrangian with respect to a K\"{a}hler form.\end{prop}
\begin{proof}  We adapt a construction from \cite{st}.  Let $\gamma:[0,1]\to \P^1$ be an embedded path missing $\{\infty,0,1\}$ with $\gamma(0)=-8$ and $\gamma(1) = \phi(-8)$.  Choose a K\"{a}hler form $\omega$ on $\hat{X}$ which agrees over a neighborhood of $\gamma$ with the sum of pullbacks of K\"{a}hler forms from each elliptic surface factor $S$ (this is explained below).  Using the horizontal distribution $\omega$-orthogonal to the $\pi$-vertical tangent spaces, there is a symplectic parallel transport along $\gamma$.    By our choice of $\omega$, the transport preserves the product structure $E\times E'$ on the fibers over $\gamma$.   Let $\ell_0\subset \pi^{-1}\left(\gamma(\frac{1}{2})\right)$ be a vanishing loop in $E$ for the flow toward 0, and $\ell_1\subset \pi^{-1}\left(\gamma(\frac{1}{2})\right)$ a vanishing loop in $E'$ for the flow toward 1.  The parallel transport along $\gamma$ sweeps out a Lagrangian:
$$L = \bigcup_{s\in [0,1]} (\ell_0)_{\gamma(s)} \times (\ell_1)_{\gamma(s)}\subset \hat{X},$$
diffeomorphic to $S^3$ fibered by 2-tori, with $S^1$ caps on either side.  It is smooth at the caps because it is locally the product of a Lefschetz thimble with $S^1$.
\end{proof}
\noindent
To show the existence of suitable K\"{a}hler forms on $\hat{X}$, we use the main idea of \cite{sty}.  Replacing $\phi$ in the construction of $X$ with a generic nearby automorphism $\phi'$ of $\P^1$, we obtain a deformation smoothing $Y$ of $X$.  Schematically,
$$\xymatrix{
 & \hat{X} \ar[d]_\eps \\
Y \ar@{~}[r] & X.
}$$
This $Y$ is commonly referred to as the {\it Schoen manifold}.  All the fibers of $Y\to \P^1$ have a smooth elliptic factor ($\nu$ and $\phi'\circ\nu$ share no critical values), so $\chi_{top}(Y)=0$.  By an argument similar to the proof of Prop. \ref{pic}, we find that $h^{1,1}(Y)=19$, so $h^{1,2}(Y)=19$ as well.  Since $Y$ is a smooth fibered product of projective surfaces, it admits a K\"{a}hler form $\omega_Y$ which is the sum of pullbacks of K\"{a}hler forms from each factor.  This form induces a closed form on $\hat{X}$ which is K\"{a}hler away from the exceptional curves.  In a slight abuse of notation, we will refer to this form as $\omega_Y$ too. \\\\
There are $n$ Lagrangian vanishing spheres $L_i\subset Y$ for the degeneration to $X$, constructed as in the proof of Prop. \ref{drag}, using the short paths from $\phi(t)$ to $\phi'(t)$ for $t\in \{\infty,0,1\}$.  The passage from $Y$ to $\hat{X}$ is called a {\it conifold transition}.  Topologically, it is one of the $2^n$ possible surgeries replacing each Lagrangian neighborhood $T^*L_i \simeq T^* S^3$ with 
$$\Tot(\O_{\P^1}(-1)\oplus \O_{\P^1}(-1)).$$
In our case, each such $\P^1$ is an exceptional curve $C_i$ of the projective small resolution $\eps:\hat{X}\to X$.  For a general conifold transition between 6-manifolds, we have:
\begin{thm} \cite{sty} \label{sty}
Fix a symplectic 6-manifold $Y$ with a collection of $n$ disjoint Lagrangian spheres $L_i$.  There exists a symplectic structure on one of the $2^n$ conifold transitions in the $L_i$ such that the resulting $C_i$ are symplectic if and only if there is a relation in homology:
$$\sum_{i=1}^n \lambda_i [L_i] = 0 \in H_3(Y,\Z),$$
with $\lambda_i\neq 0$ for all $i$.
\end{thm}
\noindent
Our conifold transition satisfies the conditions of Thm. \ref{sty} because $\hat{X}$ is projective.  The proof of the `if' part of the theorem proceeds by finding a 4-chain $\eta$ on $Y$ with boundary equal to $\cup_i L_i$, realizing the relation above.  After passing through the conifold transition surgery, $\eta$ becomes a 4-cycle whose Poincar\'{e} dual is represented by a closed 2-form $\sigma$.  Using cutoff functions and the fact that $H^2(C_i,\R)\simeq \R$, one can arrange that $\sigma$ restricts to a $J$-tame form on a tubular neighborhood $U_i$ of each exceptional curve $C_i$, where $J$ is the integrable almost complex structure on $\hat{X}$.  Since $\omega_Y$ is $J$-tame away from the $C_i$ and $\hat{X}\smallsetminus \cup_i U_i$ is compact, there exists $N\gg 0$ such that $N\omega_Y + \sigma$ is $J$-tame everywhere.  Now,
$$H^{0,2}(\hat{X}) = H^{2,0}(\hat{X})=0$$
implies that every closed 2-form has type (1,1), so $J$-tame forms are K\"{a}hler.  The key point is that the modification of $\omega_Y$ is local around $\eta$.  We will produce a 4-chain $\eta$ on $Y$ disjoint from the fibers over the path $\gamma$ in Prop. \ref{drag}.  This means that the K\"{a}hler form on $\hat{X}$ can be chosen to agree with $\omega_Y$ in a neighborhood of $\pi^{-1}(\gamma)$.\\\\
Let $g:[0,1]\to \P^1$ be a path from $\phi(t)$ to $\phi'(t)$ for $t\in \{\infty,0,1\}$.  The fiber of $Y\to \P^1$ over $g(\frac{1}{2})$ is a product of elliptic curves:  $E\times E'$.  Now, $E$ (resp. $E'$) contains $b$ (resp. $b'$) disjoint homologous loops which get pinched to nodes in the fibers by the flow $s\to 0$ (resp. $s\to 1$).  If $\ell_0$ is such a loop in $E$, $\ell_1$ and $\ell_1'$ are adjacent such loops in $E'$, and $K\subset E'$ is a cylinder with boundary $\ell_1\cup \ell_1'$, then the parallel transport along $g$ sweeps out a $4$-chain:
$$\bigcup_{s\in [0,1]} (\ell_0)_{g(s)} \times (K)_{g(s)}\subset Y.$$
Its boundary is the union of two vanishing 3-spheres and an $S^1\times S^2$ given by $(\ell_0)_{g(1)}$ cross a $\P^1$ component of $(\phi'\circ\nu)^{-1}(g(1))$.  Repeating this construction for all other choices of adjacent loops in either factor, weighted appropriately, we obtain the desired $\eta$, which is supported over the short paths from $\phi(t)$ to $\phi'(t)$ for $t\in \{\infty,0,1\}$, disjoint from $\pi^{-1}(\gamma)$.

\begin{remark}
The construction above only works because $b,b'>1$ (in our case, they are 2, 3, or 6).  Otherwise the cylinder $K$ would not have a boundary.  This observation fits with the description of the projective small resolution $\eps:\hat{X}\to X$ as a blow up.  The irreducible components of a fiber $I_b\times I_{b'}$ are non-Cartier divisors in the threefold if and only if $b,b'>1$; see Lemma 3.1 of \cite{schoen}.
\end{remark}
\noindent
Lastly, we point out that the conifold transition lets us compute the Betti numbers of $X$ using Thm. 2.11 of \cite{sty}.  If $r$ is the rank of the span of $[L_i]$ in $H_3(Y,\Z)$, then 
\begin{align*}
b_2(\hat{X}) &= b_2(X)+n-r = b_2(Y)+n-r;\\
b_3(\hat{X}) &= b_3(X) - r = b_3(Y) - 2r;\\
b_4(\hat{X}) &= b_4(X) = b_4(Y)+n-r.
\end{align*}
Since $b_3(\hat{X})=2$ and $b_3(Y)=40$, we deduce that $r=19$.

\section{Elliptic Modular Surfaces}
\noindent
The classical modular curves $X_1(N)$ are constructed by compactifying quotients of the upper half plane $\H$ by the congruence subgroups $\Gamma_1(N)\subset SL_2(\Z)$:
$$\Gamma_1(N):= \left\{ \begin{pmatrix}a & b \\ c & d \end{pmatrix}  \in SL_2(\Z) : a\equiv d\equiv 1\,(6), c\equiv 0\,(6)   \right\};$$
\begin{align*}
\H^* &:= \H \cup \P^1(\Q);\\
X_1(N) &:= \H^*/\Gamma_1(N).
\end{align*}
The action of $\Gamma_1(N)$ on $\P^1(\Q)$ has finitely many orbits, which become cusps in the modular curve.  The stabilizer of a point in $\P^1(\Q)$ is a parabolic subgroup of $\Gamma_1(N)$ generated by a conjugate of
$$\begin{pmatrix} 1 & b \\ 0 & 1 \end{pmatrix},$$
where $b\in \mathbb N$ is called the width of the corresponding cusp.  There is a universal family over $\H/\Gamma_1(N)$ whose fiber at $\tau$ is the elliptic curve $\C/(\Z+\Z\tau)$.  This family admits a compactification over $X_1(N)$ by adding a fiber of Kodaira type $I_b$ over each cusp.  For $1\leq N\leq 10$ and $N=12$, the curve $X_1(N)$ has genus 0, and the Hauptmodul $j_N:X_1(N) \to \P^1$ is an isomorphism defined over $\Q$.  For our application, we specialize to the case $N=6$ where the elliptic modular surface is isomorphic to our example $\nu:S \to \P^1$ by \cite{beauville}.\\\\
Toward our ultimate goal of describing the homology of the Calabi-Yau threefold $\hat{X}$, we record the monodromy representation associated to $\nu:S\to \P^1$.
\begin{prop}\label{mono}
For a chosen base point $*\in \P^1\smallsetminus \{-8,\infty,1,0\}$, and a star-shaped, cyclically ordered collection of paths from $*$ to $-8$, $\infty$, $1$, and $0$, the monodromy
$$\pi_1(\P^1\smallsetminus\{-8,\infty,1,0\},*) \to SL(H_1(\nu^{-1}(*),\Z)) \simeq SL_2(\Z)$$
maps the corresponding positively oriented loops $\gamma_{-8}$, $\gamma_\infty$, $\gamma_1$, $\gamma_0$ respectively to
$$M_1=\begin{pmatrix} 1 & 1 \\ 0 & 1 \end{pmatrix},\, M_6= \begin{pmatrix} 1 & 0 \\ -6 & 1 \end{pmatrix},\, M_2= \begin{pmatrix} -5 & 2 \\ -18 & 7 \end{pmatrix},\, M_3= \begin{pmatrix} -5 & 3 \\ -12 & 7 \end{pmatrix}\in \Gamma_1(6).$$
\end{prop}
\begin{proof}
The action of $\Gamma_1(6)$ on $\H$ has a fundamental domain $\mathcal G$, which can be built by translating the strip fundamental domain $\mathcal F$ for $PSL_2(\Z)$ by 12 representatives $g_i$ for the left cosets of $\Gamma_1(6)\subset PSL_2(\Z)$:
$$\mathcal G  = \bigcup_{i=1}^{12} g_i^{-1} \mathcal F.$$
The coset space is identified with the $PSL_2(\Z)$ orbit of $(1,0)$ in $(\Z/6\Z)^2/\pm 1$.  If we choose the coset representatives
$$\begin{pmatrix} 1 & 0 \\ 0 & 1 \end{pmatrix}, \begin{pmatrix} 0 & 1 \\ -1 & 0 \end{pmatrix}, \begin{pmatrix} -1 & 1 \\ -1 & 0 \end{pmatrix}, \begin{pmatrix} -2 & 1 \\ -1 & 0 \end{pmatrix}, \begin{pmatrix} -3 & 1 \\ -1 & 0 \end{pmatrix}, \begin{pmatrix} -4 & 1 \\ -1 & 0 \end{pmatrix},$$
$$ \begin{pmatrix} -5 & 1 \\ -1 & 0 \end{pmatrix}, \begin{pmatrix} -5 & 3 \\ -2 & 1 \end{pmatrix}, \begin{pmatrix} -7 & 4 \\ -2 & 1 \end{pmatrix}, \begin{pmatrix} -9 & 5 \\ -2 & 1 \end{pmatrix}, \begin{pmatrix} -11 & 4 \\ -3 & 1 \end{pmatrix}, \begin{pmatrix} -14 & 5 \\ -3 & 1 \end{pmatrix},$$
then $\mathcal G$ is contractible and has cusps at $\tau = i\infty, 0, \frac{1}{3}, \frac{1}{2}$.  The universal family of elliptic curves over $\mathcal G$ is topologically trivial, so for any base point $\tau_0$ in the interior of $\mathcal G$, we can take the standard basis for $H_1(\nu^{-1}(\tau_0),\Z)\simeq \Z^2$.  Choose non-intersecting paths in $\mathcal G$ from $\tau_0$ limiting to each of the four cusps.  The quotient $\mathcal G/\Gamma_1(6)$ is homeomorphic to a sphere with four punctures, and $\pi_1(\mathcal G/\Gamma_1(6),\tau_0)$ is generated by four oriented loops around the punctures, conjugated by the paths above, with the relation that their cyclically ordered product is 1.  The $\Gamma_1(6)$-stabilizer of each cusp of $\mathcal G$ is generated by the respective matrix in the statement of the proposition, and this gives the monodromy action on $\Z^2$.  The isomorphism $j_6: X_1(6)\to \P^1$ must identify the cusps with the critical values of the pencil $\nu$ matching cusp widths with Kodaira types, so $j_6(i\infty)=-8$, $j_6(0)=\infty$, $j_6(\frac{1}{3})=1$, and $j_6(\frac{1}{2})=0$.  We set $*=j(\tau_0)$.
\end{proof}

\section{Homology class of $L$}\label{mv}
\noindent
To show that the Lagrangian sphere $L\subset \hat{X}$ is essential (for all choices of $\phi$ and $\gamma$), we construct a second Lagrangian sphere that intersects it non-trivially.  Recall that $\gamma$ is an embedded path from $-8$ to $\phi(-8)$ missing the other critical values, and let $U\subset \P^1$ be a small neighborhood of $\gamma$.  Let $V\subset \P^1$ be an open set containing another critical value $c$ with fiber $I_b\times I_{b'}$ such that $b\neq b'$, overlapping with $U$ as pictured below in the case where $b=2$ and $b'=6$.  Such a value exists because $\phi$ permutes $\{\infty,0,1\}$ non-trivially.  Let $\gamma'$ be another path from $-8$ to $\phi(-8)$ embedded in $U\cup V \smallsetminus \{c\}$, not homotopic to $\gamma$ in that complement, and meeting $\gamma$ only at $-8$ and $\phi(-8)$.  The construction of Prop. \ref{drag} applied to $\gamma'$ gives a second Lagrangian sphere $L'$.  Note that as a set,
$$L\cap L' \subset \pi^{-1}(-8) \cup \pi^{-1}(\phi(-8)).$$
\begin{prop}\label{mult}
The intersection product $[L]\cdot [L']$ is a nonzero multiple of $6$.
\end{prop}
\begin{proof}
Choose a basis for homology of the smooth fiber $H_1(E\times E',\Z) \simeq \Z^2\oplus \Z^2$ over an interior point of $\gamma$ such that the vanishing loops $\ell_0\subset E$ and $\ell_1\subset E'$ each have class $\begin{bmatrix}1\\0  \end{bmatrix}$ in their respective copy of $\Z^2$.  The intersection $L\cap \pi^{-1}(\phi(-8))$ is the parallel transported loop $\ell_0$ in the smooth elliptic factor, cross the node in the $I_1$ factor.  The intersection $L'\cap \pi^{-1}(\phi(-8))$ is a different loop in the smooth elliptic factor, cross the node in the $I_1$ factor.  Hence, the contribution to $[L]\cdot [L']$ from points in $\pi^{-1}(\phi(-8))$ can be computed on the smooth elliptic factor $E$.  Likewise, the contribution to $[L]\cdot [L']$ from $\pi^{-1}(-8)$ can be computed on the smooth elliptic factor $E'$.  In terms of the (conjugated) monodromy matrices from Prop. \ref{mono},
$$[L]\cdot [L'] = \begin{bmatrix} 1\\ 0 \end{bmatrix} \cdot gM_bg^{-1}\begin{bmatrix} 1\\ 0 \end{bmatrix} + \begin{bmatrix} 1\\ 0 \end{bmatrix} \cdot h^{-1}M^{-1}_{b'}h\begin{bmatrix} 1\\ 0 \end{bmatrix}.$$
Here $g,h\in \Gamma_1(6)$ come from the ambiguity in choosing the paths $\gamma$ and $\gamma'$ in the punctured sphere.  Each summand above is equal to the lower left entry in a monodromy matrix.  Observe that the matrices $M_b$ ($b=2,3,6$) in Prop. \ref{mono} have lower left entries $-36/b$.  Conjugating $M_b$ by an element of $\Gamma_1(6)$ does not change the lower left entry modulo 36.  Therefore, the difference of any two such entries for $b\neq b'$ is a nonzero multiple of 6.
\end{proof}
$$\includegraphics[width=102mm]{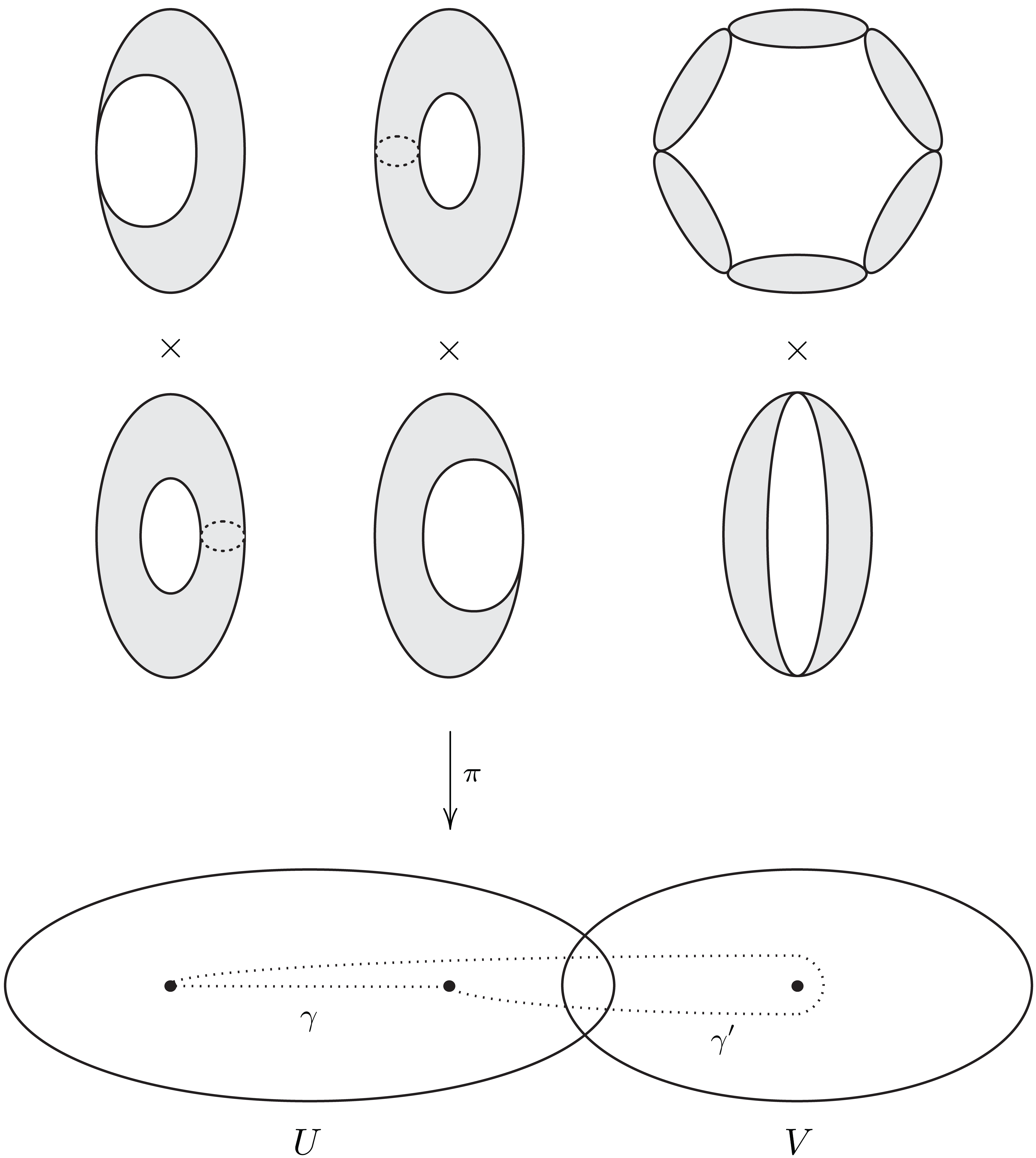}$$

\section{Degenerations of $\hat{X}$}
\noindent
Suppose that $\hat{X}$ admits a K\"{a}hler degeneration.  That is, $\hat{X}$ is isomorphic to a fiber of a proper holomorphic family over a complex disk
$$f:\mathcal X \to \Delta,$$
and $f^{-1}(0)$ is singular.  By Prop. \ref{rigid}, $\hat{X}$ has no moduli so $f$ is holomorphically locally trivial away from $0\in \Delta$, by the Fischer-Grauert theorem.  The fiber bundle 
$$f^{-1}(\Delta^*)  \to \Delta^*$$
has monodromy valued in $Aut(\hat{X})$.  In other words,
$$\Z\simeq \pi_1(\Delta^*) \to Aut(\hat{X}) \to Sp( H_3(\hat{X},\Z ))\simeq SL_2(\Z).$$
To control potential isotrivial degenerations, we prove that $Aut(\hat{X})$ is finite.
\begin{prop}
Every automorphism preserves the fibration $\pi:\hat{X}\to \P^1$.
\end{prop}
\begin{proof}
Let $\varphi:\hat{X}\to \hat{X}$ be an automorphism, and let $A$ be a general fiber of $\pi$.  If $\varphi$ does not preserve the fibration, then $\varphi(A)$ surjects onto $\P^1$.  The image of $\varphi(A)$ in $S$ cannot be all of $S$ because complex tori only surject onto projective spaces and complex tori \cite{demailly}.  Thus, $\phi(A)$ maps onto a curve $C$ in $S$, and the generic fiber is an elliptic curve.  There must be singular fibers because $C$ surjects to $\P^1$, which contradicts the fact that $\chi_{top}(A)=0$.
\end{proof}
\begin{prop}
The group $Aut(\hat{X})$ is isomorphic to $(\Z/6\Z)^2$.
\end{prop}
\begin{proof}
Recall that $\pi$ has three non-isomorphic singular fibers, so any $\varphi\in Aut(\hat{X})$ satisfies $\pi = \pi\circ \varphi$.  This implies that $Aut(\hat{X}) = Aut(X_\eta) \simeq Aut(S_\eta) \times Aut(S_\eta)$, since the generic fiber is a product of non-isogenous elliptic curves.  As an elliptic curve over $K\simeq \C(t)$ with non-constant $j$-invariant, $S_\eta$ has automorphism group isomorphic to its group of $K$-points (Mordell-Weil group), which is $\Z/6\Z$ by \cite{mirandabasic}.
\end{proof}
\noindent
Therefore, the image of monodromy is a finite subgroup of $SL_2(\Z)$, so it is abelian of order $\leq 6$.  In particular, $L$ is not the vanishing cycle of a nodal degeneration.

\bibliographystyle{plain}
\bibliography{edonaldson}

\end{document}